\documentclass[12pt,reqno]{amsart}   	
\usepackage{bm} 
\usepackage[T1]{fontenc}
\usepackage{amsfonts}
\usepackage{mathrsfs}
\usepackage{yfonts}
\usepackage{url}
\usepackage[osf]{mathpazo}  

\usepackage{mathtools}
\usepackage{amssymb}  		
\usepackage{graphicx}
\usepackage{verbatim}
\usepackage{enumerate}	
\usepackage[english]{babel}
\usepackage{amsmath}
\usepackage{bussproofs}	
\usepackage{quoting}
\usepackage{array,booktabs}

\usepackage[shortlabels]{enumitem}
\usepackage{units}
\usepackage{xspace}\xspaceaddexceptions{\,}
\usepackage{multicol}
\usepackage{booktabs,caption}
   \usepackage{tikz-cd}
\quotingsetup{font=small}	
\usepackage{amsthm}								

\newcommand*\circled[1]{\tikz[baseline=(char.base)]{
            \node[shape=circle,draw,inner sep=4pt] (char) {#1};}}

\usepackage[pdftex,bookmarks,bookmarksnumbered,linktocpage,   
         colorlinks,linkcolor=blue,citecolor=blue]{hyperref}

\theoremstyle{definition}
\newtheorem{definition}{Definition}
\newtheorem{example}[definition]{Example}

\newtheorem{remark}[definition]{Remark}

\theoremstyle{plain}
\newtheorem{lemma}[definition]{Lemma}

\newtheorem{theorem}[definition]{Theorem}
\newtheorem{corollary}[definition]{Corollary}

\newcommand\A{{\mathbf A}}
\newcommand\B{{\mathbf B}}
\newcommand\C{{\mathbf C}}

\newcommand\Fm{\mathbf{Fm}}

\newcommand\two{\mathbf B_2}
\newcommand\WK{{\mathbf{WK}}}

\newcommand\M{\mathrm{\mathbf{M}_4}}

\newcommand\CL{\ensuremath{\mathrm{CL}}\xspace}

\newcommand\PWK{\ensuremath{\mathrm{PWK}}\xspace}

\newcommand\pwk{{\mathsf{PWK}}}

\bmdefine{\boldstar}{\mathchoice{\textstyle*}{\textstyle*}{\textstyle*}{\scriptstyle*}}
\newcommand\Modstar{\mathsf{Mod}^{\boldstar}}
\newcommand\Algstar{\mathsf{Alg}^{\boldstar}}
\newcommand{\ModS}{\mathsf{Mod}^{\textup{Su}}}

\newcommand\Alg[1]{\if#1*\operatorname{\mathsf{Alg}*}\else\operatorname{\mathsf{Alg}}#1\fi}

\newcommand\Mod[1]{\if#1*\operatorname{\mathsf{Mod}*}\else\operatorname{\mathsf{Mod}}#1\fi}

\bmdefine{\Leibniz}{\Omega}        

\bmdefine{\frege}{\Lambda}         

\makeatletter
\newcommand{\tarskidsp}{\mathord%
   {\m@th\raisebox{0pt}[0pt][0pt]{$\stackrel%
   {\raisebox{-2.7pt}[0ex][0pt]{$\displaystyle \,\?\thicksim$}}%
   {\displaystyle\Leibniz}$}}}
\newcommand{\tarskitxt}{\mathord%
   {\m@th\raisebox{0pt}[0pt][0pt]{$\stackrel%
   {\raisebox{-2.7pt}[0ex][0pt]{$\,\?\thicksim$}}{\displaystyle\Leibniz}$}}}
\newcommand{\tarskiscr}{\mathord%
   {{\m@th\raisebox{0pt}[0pt][0pt]{$\stackrel%
   {\raisebox{-2.4pt}[0ex][0pt]{$\scriptstyle \,\?\thicksim$}}%
   {\scriptstyle\Leibniz}$}}}}
\newcommand{\tarskiscrscr}{\mathord%
   {{\m@th\raisebox{0pt}[0pt][0pt]{$\stackrel%
   {\raisebox{-2pt}[0ex][0pt]{$\scriptscriptstyle \,\?\thicksim$}}%
   {\scriptscriptstyle\Leibniz}$}}}}
\newcommand{\Tarski}{\@ifnextchar ^ %
   {\mathchoice{\tarskidsp\kern-.07em}{\tarskitxt\kern-.07em}%
   {\tarskiscr\kern-.07em}{\tarskiscrscr\kern-.07em}}%
   {\mathchoice{\tarskidsp}{\tarskitxt}{\tarskiscr}{\tarskiscrscr}}}
\makeatother
\DeclareMathAlphabet{\mathbfsf}{\encodingdefault}{\sfdefault}{bx}n
\makeatletter
\providecommand*{\Dashv}{\mathrel{\mathpalette\@Dashv\vDash}}
\newcommand*{\@Dashv}[2]{\reflectbox{$\m@th#1#2$}}
\makeatother
\newcommand\pair[1]{{\langle#1\rangle}}

\useshorthands{"}
\defineshorthand{"-}{{}\nobreakdash-\hspace{0pt}}	

\newcommand{\FFi}{\mathcal{F}i}

\newcommand\PL{{\mathcal{P}}_{\text{\l}}}

\newcommand\ant{\nicefrac12}
\EnableBpAbbreviations

\newcommand{\bit}{\begin{itemize}}    
\newcommand{\eit}{\end{itemize}}
\newcommand{\ben}{\begin{enumerate}}
\newcommand{\een}{\end{enumerate}}

\newcommand{\benroman}{\ben[\normalfont (i)]}  
\let\eroman\een

\newcommand{\bde}{\begin{description}}
\newcommand{\ede}{\end{description}}

\newcommand{\Var}{\mathnormal{V\mkern-.8\thinmuskip ar}} 
\newcommand{\?}{\ensuremath{\mkern0.4\thinmuskip}}   
\newcommand{\sineq}{\mathrel{\dashv\mkern1.5mu\vdash}}  

 \usepackage{amsaddr}

\begin{document}

\title[Containment logics]{Containment logics: algebraic completeness and axiomatization}

\date{}

\author{Stefano Bonzio}
\address{University of Turin, Italy}
\email{stefano.bonzio@gmail.com}
\author{Michele Pra Baldi}
\address{University of Cagliari, Italy}
\maketitle


%
%

\begin{abstract}

The paper studies the containment companion of a logic $\vdash$. This consists of the consequence relation $\vdash^{r}$ which satisfies all the inferences of $\vdash$, where the variables of the conclusion are \emph{contained} into those of the (set of) premises. In accordance with the work started in \cite{BonzioMorascoPrabaldi}, we show that a different generalization of the  P\l onka sum construction, adapted from algebras to logical matrices, allows us to provide a matrix-based semantics for containment logics. In particular, we provide an appropriate completeness theorem for a wide family of containment logics, and we show how to produce a complete Hilbert style axiomatization.
\end{abstract}

\section{introduction}

It is a recent discovery (see \cite{Bonzio16}) that the algebraic counterparts of weak Kleene logics are formed by a regularized variety, whose members coincide with the \emph{P\l onka sum} of Boolean algebras, the algebraic semantics of propositional classical logic. Subsequently, the abstract construction of the P\l onka sum of algebras has been generalized to logical matrices \cite{BonzioMorascoPrabaldi}. The main outcome is that the suggested notion provides an algebra-based semantics for a class of propositional logics, called \emph{logics of left variable inclusion}, of which paraconsistent weak Kleene represents the most prominent example. 

The logics in the weak Kleene ``family'' -- essentially, Bochvar \cite{BochvarBergmann} and paraconsistent weak Kleene \cite{Hallden} --  are syntactically characterized by imposing certain limitations on the inclusions of variables to classical propositional logic \cite{Urquhart2001,CiuniCarrara}.   

The extension of the construction of P\l onka sums to logical matrices, introduced in \cite{BonzioMorascoPrabaldi}, allows for an insightful investigation into the algebraic features of those logics, where the inclusion of variables runs from premises to conclusions. However, this is just one side of the coin of the \emph{logics of variable inclusion}; the other side consisting of those logics verifying inferences in which the variables occurring in the conclusion are contained into the ones occurring in the premises. Consequence relations satisfying this feature are usually know as \emph{containment logics}; the syntactic requirement which they share is a strengthen form of what Ferguson \cite{ferguson2017meaning} understands as \emph{Proscriptive Principle}, which also resembles the one defining logics of \emph{analytic containment} \cite{ledda2019algebraic, epstein1990semantic}.

The most famous example of containment logic is \emph{Bochvar logic} $\mathsf{B_{3}}$ \cite{BochvarBergmann}, defined by a single matrix which features the presence of an infectious truth-value (a peculiarity shared by the twin-sister paraconsistent weak Kleene). $\mathsf{B_{3}}$ has been successfully applied in different contexts: avoiding paradoxes in set-theory \cite{BochvarBergmann}, modeling computer programs affected by errors \cite{Ferguson} and non-sensical information databases \cite{Ciuni1}, capturing the notion of truth in relations with on/off topic arguments \cite{beall2016off}.

The main condition defining containment logics mirrors the syntactic requirement defining logics of left variable inclusion. For this reason, the present paper may be understood as an ideal continuation of the path started in \cite{BonzioMorascoPrabaldi}.
This amounts to answer the very natural question on whether it is possible to build a new generalization of the P\l onka sum construction, suitable for obtaining a matrix semantics for containment logics. Due to the intrinsic difference between the  variable inclusion constraints at stake, the solution is not self-evident. 

The paper is structured as follows.

In Section \ref{sec: preliminari}, we recall all the preliminary notions needed to go trough the reading of the whole paper. They basically consist of the basic notions of abstract algebraic logic and of the theory of P\l onka sums.
 
In Section \ref{sec:compl}, we formally introduce containment logics. By providing an adequate notion of P\l onka sum for logical matrices, we obtain soundness and completeness for the \emph{containment companion} $\vdash^{r}$ of an arbitrary (finitary) logic $\vdash$, with respect to the P\l onka sum of the matrix models of $\vdash$.

In Section \ref{sec:hilbert}, we focus on a specific (though very wide) class of logics, namely those  possessing a binary term called partition function (a property shared by the vast majority of known logics). We provide a method for obtaining an Hilbert-style axiomatization for a logic $\vdash^{r}$ (Theorem \ref{th: completezza calcolo Hilbert}) out of an axiomatization for (a finitary) logic $\vdash$. 

Finally, in Section \ref{sec: examples}, we put at work our machinery and characterize the axiomatization of containment companions of some well-known logics, namely of classical propositional logic, Belnap-Dunn and the Logic of Paradox.


%
%
\section{Preliminaries}\label{sec: preliminari}


For standard background on universal algebra and abstract algebraic logic, we refer the reader, respectively, to \cite{Be11g,BuSa00} and \cite{Font16}. In this paper, algebraic languages are assumed not to contain constant symbols. 
Moreover, unless stated otherwise, we work within a fixed but arbitrary algebraic language. We denote algebras by $\A, \B, \C\dots$ respectively with universes $A, B, C \dots$. 
Let $\Fm$ be the algebra of formulas built up over a countably infinite set $\Var$ of variables (which we indicate by $x,y,z,\dots$). Given a formula $\varphi\in Fm$, we denote by $\Var(\varphi)$ the set of variables really occurring in $\varphi$. Similarly, given $\Gamma\subseteq Fm$, we set
\[
\Var(\Gamma)=\bigcup \{\Var(\gamma)\colon \gamma\in\Gamma\}.
\]
A \emph{logic} is a substitution invariant consequence relation $\vdash \?\? \subseteq \mathcal{P}(Fm) \times Fm$ meaning that for every substitution $\sigma \colon \Fm \to \Fm$,
\[
\text{if }\Gamma \vdash \varphi \text{, then }\sigma [\Gamma] \vdash \sigma (\varphi).
\]
Given formulas $\varphi, \psi$, we write $\varphi \sineq \psi$ as a shorthand for $\varphi \vdash \psi$ and $\psi \vdash \varphi$. 
A logic $\vdash$ is \emph{finitary} when for all $\Gamma\cup\varphi\subseteq Fm$:
\begin{align*}
\Gamma\vdash\varphi \Longleftrightarrow \exists \Delta \subseteq\Gamma \text{ such that } \Delta \text{ is finite and } \Delta\vdash\varphi.
\end{align*}

 A \emph{matrix} is a pair $\langle \A, F\rangle$ where $\A$ is an algebra and $F \subseteq A$. In this case, $\A$ is called the \textit{algebraic reduct} of the matrix $\langle \A, F \rangle$. 
 
 
 Every class of matrices $\mathsf{M}$ defines a logic as follows:
\begin{align*}
\Gamma \vdash_{\mathsf{M}} \varphi \Longleftrightarrow& \text{ for every }\langle \A, F \rangle \in \mathsf{M} \text{ and homomorphism }h \colon \Fm \to \A,\\
& \text{ if }h[\Gamma] \subseteq F\text{, then }h(\varphi) \in F.
\end{align*}
We say that a logic $\vdash$ is \emph{complete} with respect to a class of matrices $\mathsf{M}$ when  $\vdash_{\mathsf{M}}\;=\;\vdash$.
Sometimes, we will refer to such homomorphisms $h$ as \emph{evaluations}.

 A matrix $\langle \A, F\rangle$ is a \emph{model} of a logic $\vdash$ when
\begin{align*}
\text{if }\Gamma \vdash \varphi, &\text{ then for every homomorphism }h \colon \Fm \to \A,\\  
&\text{ if }h[\Gamma] \subseteq F\text{, then }h(\varphi) \in F.
\end{align*}
A set $F \subseteq A$ is a (deductive) \textit{filter} of $\vdash$ on $\A$, or simply a $\vdash$-\textit{filter}, when the matrix $\langle \A, F \rangle$ is a model of $\vdash$. We denote by $\FFi_{\vdash}\A$ the set of all filters of $\vdash$ on $\A$.

 Although the present paper does not address the study of reduced models (for containment logics), in order to make it self-contained, we recall those notions, concerning reduced models, that will be used. Let $\A$ be an algebra and $F \subseteq A$. A congruence $\theta$ of $\A$ is \emph{compatible} with $F$ when for every $a,b\in A$,
\[
\text{if }a\in F\text{ and }\langle a, b \rangle \in \theta\text{, then }b\in F.
\]
The largest congruence of $\A$ which is compatible with $F$ always exists, and is called the \emph{Leibniz congruence} of $F$ on $\A$. It is denoted by $\Leibniz^{\A}F$. 
The \emph{Suszko congruence} of $F$ on $\A$, is defined as
\[
\Tarski^{\A}_{\vdash}F \coloneqq \bigcap \{ \Leibniz^{\A}G : F \subseteq G \text{ and }G \in \FFi_{\vdash}\A \}.
\]

The Leibniz and Suszko congruences allow to single out a distinguished class of models of logics. More precisely, given a logic $\vdash$, we set
\begin{align*}
\Mod(\vdash) & \coloneqq \{ \langle \A, F \rangle : \langle \A, F \rangle \text{ is a model of }\vdash \};\\
\Modstar(\vdash) & \coloneqq \{ \langle \A, F \rangle \in \Mod(\vdash) : \Leibniz^{\A}F \text{ is the identity} \};\\
\ModS(\vdash) & \coloneqq \{ \langle \A, F \rangle \in \Mod(\vdash) : \Tarski_{\vdash}^{\A}F \text{ is the identity} \}.
\end{align*}
The above classes of matrices are called, respectively, the classes of \text{models}, \textit{Leibniz reduced models} (or, simply reduced models), and \textit{Suszko reduced models} of $\vdash$. 

Given a logic $\vdash$, we set
\[
\Algstar(\vdash) = \{ \A : \text{ there is }F \subseteq A \text{ s.t. }\langle \A, F \rangle \in \Modstar(\vdash) \},    \text{ and}
\]
\[
\Alg(\vdash) = \{ \A : \text{ there is }F \subseteq A \text{ s.t. }\langle \A, F \rangle \in \ModS(\vdash) \}.
\]

$\Alg(\vdash)$ is the class of algebraic reducts of matrices in $\ModS(\vdash)$. 
The class $\Alg(\vdash)$ is called the \textit{algebraic counterpart} of $\vdash$ as, for the vast majority of logics $\vdash$, $\Alg(\vdash)$  is the class of algebras intuitively associated with $\vdash$.


Trivial matrices have a central role in the whole paper. We say that a matrix $\langle \A, F \rangle$ is \textit{trivial} if $F = A$. We denote by $\langle \boldsymbol{1}, \{ 1 \} \rangle$ the trivial matrix, whose algebraic reduct $\boldsymbol{1}$ is the trivial algebra.
Observe that the latter matrix is a model (resp. reduced, Suszko reduced model) of every logic. Moreover, if $\vdash$ is a logic and $\langle \A, F\rangle \in \ModS(\vdash)$ is a trivial matrix, then $\langle \A, F\rangle = \langle \boldsymbol{1}, \{ 1 \}\rangle$.

A set of models of a logic $\vdash$ is said to be non trivial, if it does not contain trivial matrices. We indicate by $\Mod_{+}(\vdash)$ the set of non trivial models of a logic $\vdash$.

\subsection*{P\l onka sums}

As standard references on P\l onka sums we mention \cite{Plo67,Plo67a,Romanowska92}. A \textit{semilattice} is an algebra $\A = \langle A, \lor\rangle$, where $\lor$ is a binary associative, commutative and idempotent operation. Given a semilattice $\A$ and $a, b \in A$, we set
\[
a \leq b \Longleftrightarrow a \lor b = b.
\]
It is easy to see that $\leq$ is a partial order on $A$.
\begin{definition}\label{Def:Directed system of algebras}
A \textit{direct system of algebras} consists of: 
\benroman
\item a semilattice $I = \langle I, \lor\rangle$;
\item a family of similar algebras $\{ \A_{i} : i \in I \}$ with pair-wise disjoint universes;
\item a homomorphism $f_{ij} \colon \A_{i} \to \A_{j}$, for every $i, j \in I$ such that $i \leq j$.
\eroman
Moreover, $f_{ii}$ is the identity map for every $i \in I$, and $f_{ik} = f_{jk} \circ f_{ij}$, for $i \leq j \leq k$.
\end{definition}

Let $X$ be a direct system of algebras as defined above. The \textit{P\l onka sum} of $X$, in symbols $\PL(X)$ or $\PL(\A_{i})_{i \in I}$\footnote{When no confusion shall occur, we will write $\PL(\A_{i})$ instead of $\PL(\A_{i})_{i\in I}$.}, is the algebra in the same type defined as follows: the universe of $\PL(\A_{i})_{i \in I}$ is the union $\bigcup_{i \in I}A_{i}$. Moreover, for every $n$-ary basic operation $g$ and $a_{1}, \dots, a_{n} \in \bigcup_{i \in I}A_{i}$, we set
\[
g^{\PL(\A_{i})_{i \in I}}(a_{1}, \dots, a_{n}) \coloneqq g^{\A_{j}}(f_{i_{1} j}(a_{1}), \dots, f_{i_{n} j}(a_{n})),
\]
where $a_{1} \in A_{i_{1}}, \dots, a_{n} \in A_{i_{n}}$ and $j = i_{1} \lor \dots \lor i_{n}$.\ 

Observe that if in the above display we replace $g$ by any complex formula $\varphi$ in $n$-variables, we still have that
\[
\varphi^{\PL(\A_{i})_{i \in I}}(a_{1}, \dots, a_{n}) = \varphi^{\A_{j}}(f_{i_{1} j}(a_{1}), \dots, f_{i_{n} j}(a_{n})).
\]
\noindent \textbf{Notation:} Given a formula $\varphi$, we will often write $\varphi^{\PL}$ instead of $\varphi^{\PL(\A_{i})_{i \in I}}$ when no confusion shall occur. 

\vspace{5pt}

The theory of P\l onka sums is strictly related with a special kind of binary operation, called partition function.

\begin{definition}\label{def: partition function}
Let $\A$ be an algebra of type $\nu$. A function $\cdot\colon A^2\to A$ is a \emph{partition function} in $\A$ if the following conditions are satisfied for all $a,b,c\in A$, $ a_1 , ..., a_n\in A $ and for any operation $g\in\nu$ of arity $n\geqslant 1$.
\begin{enumerate}[label=\textbf{P\arabic*}., leftmargin=*]
\item $a\cdot a = a$;
\item $a\cdot (b\cdot c) = (a\cdot b) \cdot c $;
\item $a\cdot (b\cdot c) = a\cdot (c\cdot b)$;
\item $g(a_1,\dots,a_n)\cdot b = g(a_1\cdot b,\dots, a_n\cdot b)$;
\item $b\cdot g(a_1,\dots,a_n) = b\cdot a_{1}\cdot_{\dots}\cdot a_n $.
\end{enumerate}
\end{definition}

Different definitions of partition function appeared in literature. We adopted the one which uses the minimal number of defining conditions (see \cite{Romanowska92}). 

The next result underlines the connection between P\l onka sums and partition functions:

\begin{theorem}\cite[Thm.~II]{Plo67}\label{th: Teorema di Plonka}
Let $\A$ be an algebra of type $\nu$ with a partition function $\cdot$. The following conditions hold:
\begin{enumerate}
\item $A$ can be partitioned into $\{ A_{i} : i \in I \}$ where any two elements $a, b \in A$ belong to the same component $A_{i}$ exactly when
\[
a= a\cdot b \text{ and }b = b\cdot a.
\]
Moreover, every $A_{i}$ is the universe of a subalgebra $\A_{i}$ of $\A$.
\item The relation $\leq$ on $I$ given by the rule
\[
i \leq j \Longleftrightarrow \text{ there exist }a \in A_{i}, b \in A_{j} \text{ s.t. } b\cdot a =b
\]
is a semilattice order. 
\item For all $i,j\in I$ such that $i\leq j$ and $b \in A_{j}$, the map $f_{ij} \colon A_{i}\to A_{j}$, defined by the rule $f_{ij}(x)= x\cdot b$ is a homomorphism. The definition of $f_{ij}$ is independent from the choice of $b$, since $a\cdot b = a\cdot c$, for all $a\in A_i$ and $c\in A_j$.
\item $Y = \langle \langle I, \leq \rangle, \{ \A_{i} \}_{i \in I}, \{ f_{ij} \! : \! i \leq j \}\rangle$ is a direct system of algebras such that $\PL(Y)=\A$.
\end{enumerate}
\end{theorem}

The statement of Theorem \ref{th: Teorema di Plonka} displayed above relies on the assumption that the algebraic language contains no constant symbols\footnote{When considering types containing constants, then additional conditions should be added to the definition of partition function. This results into a decomposition over a semilattice having a least element: constants of the P\l onka sum will belong to the algebra whose index is the least element. For further details, see \cite{plonka1984sum}.}. 
It is worth remarking that the construction of Plonka sums preserves the validity of \textit{regular identities}, i.e. identities of the form $ \varphi \thickapprox \psi $ such that $\Var(\varphi) = \Var(\psi)$.

\section{Algebraic completeness}\label{sec:compl}

The usual presentations of Kleene three-valued logics divide them into two families, depending on the meaning given to the connectives $\land,\lor$: \emph{strong logics} -- including Strong Kleene and the Logic of Paradox \cite{Priestfirst} -- and \emph{weak logics} -- Bochvar logic ($\mathrm{B_3}$) and paraconsistent weak Kleene (or Halld\'en logic). Logics in each family differentiate upon the choice of the truth-set: $\{1\}$ in Strong Kleene and Bochvar, $\{1,\ant \}$ in the logic of Paradox and paraconsistent weak Kleene.

Bochvar \cite{BochvarBergmann} is the logic induced by the matrix $\pair{\WK,\{1\}}$ of the so-called weak Kleene tables\footnote{In accordance with \cite{Urquhart2001}, here we are actually considering Bochvar ``internal calculus'', which is only one of the two logics introduced in \cite{BochvarBergmann}. The ``external calculus'' consists of a linguistic extension of the weak Kleene tables, with a connective $\mathrm{t}$, interpreted (for every evaluation $h$) as $h(\mathrm{t}\varphi)= 1$ if and only if $h(\varphi)=1$ (for further details, see \cite{BochvarBergmann,KarpenkoTomova,FinnGrigolia}).} displayed in Figure \ref{fig: tavole weak}. 

\vspace{5pt}
\begin{figure}[h]
\begin{center}\renewcommand{\arraystretch}{1.25}
\begin{tabular}{>{$}c<{$}|>{$}c<{$}>{$}c<{$}>{$}c<{$}}
   \land & 0 & \ant & 1 \\[.2ex]
 \hline
       0 & 0 & \ant & 0 \\
       \ant & \ant & \ant & \ant \\          
       1 & 0 & \ant & 1
\end{tabular}
\qquad
\begin{tabular}{>{$}c<{$}|>{$}c<{$}>{$}c<{$}>{$}c<{$}}
   \lor & 0 & \ant & 1 \\[.2ex]
 \hline
     0 & 0 & \ant & 1 \\
     \ant & \ant & \ant & \ant \\          
     1 & 1 & \ant & 1
\end{tabular}
\qquad
\begin{tabular}{>{$}c<{$}|>{$}c<{$}}
  \lnot &  \\[.2ex]
\hline
  1 & 0 \\
  \ant & \ant \\
  0 & 1 \\
\end{tabular}

\caption{The algebra $\mathbf{WK}$ of weak Kleene tables.}\label{fig: tavole weak}
\end{center}
\end{figure}
\vspace{10pt}
\noindent
It is not difficult to check that the algebra $\mathbf{WK}$ is the P\l onka sum of the two-element Boolean algebra $\two$ and the trivial (Boolean) algebra $\mathbf{\ant}$ (over the index set given by the two-element semilattice)\footnote{We refer the reader interested in further details directly to \cite{Bonzio16}.}.

Bochvar logic can be equivalently presented as follows:

\begin{theorem}\cite[Theorem 2.3.1]{Urquhart2001}\label{teorema: Urquhart}
The following are equivalent: 
\begin{enumerate}
\item $\Gamma\vdash_{\mathrm{B_3}}\varphi$;
\item $\Gamma\vdash_{\CL}\varphi $ with $ \Var(\varphi)\subseteq\Var(\Gamma)$ or $\Gamma$ is an inconsistent set of formulas.  
\end{enumerate}
\end{theorem}
 
In words, Bochvar logic is the consequence relation obtained out of classical logics, imposing the constraint that variables of the conclusion (formula) shall be included into those of the (set of) premises. For this reason, $\mathrm{B_{3}}$ is often referred to as a \emph{containment logic}, see for e.g. \cite{ferguson2017meaning,Parry}.

The following definition originates in \cite{ThomasLavtesi}, but see also \cite{CampercholiRaftery,JGR13}.

\begin{definition}\label{def inconsistency terms}
A set of formulas $\Sigma$ is an antitheorem of a logic $\vdash$ if
$\sigma[\Sigma] \vdash \varphi$, for every substitution $\sigma:\Fm\to\Fm$ and formula $\varphi$.
\end{definition}

Observe that, if the set $\Sigma(y_{1},\dots,y_{n})$, where the variables $y_{1},\dots,y_{n}$ really occur, is an antitheorem for $\vdash$, then, by substitution, also $\Sigma(x)$ (where only $x$ occurs) is an antitheorem for $\vdash$. In other words, if a logic $\vdash$ possesses an antitheorem $\Sigma$, then it possesses an antitheorem in one variable only and we will write $\Sigma(x)$.  

The most intuitive example one can keep in mind is the following: for any formula $\varphi$, the set $\{\varphi,\neg\varphi\}$ is an antitheorem of Intuitionistic, Classical and both local and global modal logics.


The above characterization of Bochvar logic suggests that a logic $\vdash^{r}$ satisfying an analogous criterion on the inclusion of variables as that of Theorem \ref{teorema: Urquhart} can be associated to any arbitrary logic $\vdash$. 

\begin{definition}\label{def: containment companion}
Let $\vdash$ be a logic. $\vdash^{r}$ is the logic defined as 

\[
\Gamma\vdash^{r}\varphi\iff \left\{ \begin{array}{ll}
\Gamma\vdash\varphi  \ \text{and} \ \Var(\varphi)\subseteq\Var(\Gamma)& \text{or}\\
\Sigma(x)\subseteq\Gamma, & \\
  \end{array} \right.  
\]

where $\Sigma(x)$ is an antitheorem of $\vdash$.
\end{definition}

We will refer to $\vdash^{r}$ as the \emph{containment companion} of the logic $\vdash$. It follows from the definition that $\vdash^{r}$ and $\vdash$ have the same antitheorems. 
Bochvar logic is not the unique example of containment logic that can be found in literature. Indeed, the four-valued logic $\mathrm{\mathbf{S}_{fde}}$, introduced by Deutsch \cite{Deutsch} (see also \cite{BelikovPetrukhin, Szmuch2016, ciunilast}), can be counted as the containment companion of the Logic of Paradox (this follows from our analysis, see Subsection \ref{Subsec: relevance}). Also one of the four-valued logics introduced by Tomova (see Example \ref{ex: containment di PWK}) is a containment logic, more precisely the containment companion of \PWK.

The containment companion ($\vdash^{r}$) of a logic $\vdash$ is, somehow, ``opposite'' to the logic $\vdash_{l}$ considered in \cite{BonzioMorascoPrabaldi}, as the inclusion of variables works from conclusions to premises,  namely from ``right to left'', and not vice-versa (as for $\vdash_{l}$).
\begin{lemma}\label{lemma: finitarieta di r}
Let $\vdash$ be a finitary logic. Then $\vdash^{r}$ is finitary.
\end{lemma}

\begin{proof}
Suppose that $\Gamma\vdash^{r}\varphi$. By definition of $\vdash^{r}$, two cases are to be considered: 
\begin{itemize}
\item[i)] $\Gamma\vdash\varphi$ with $\Var(\varphi)\subseteq\Var(\Gamma)$;
\item[ii)] $\Gamma $ contains an antitheorem $\Sigma(x)$ of $\vdash$.
\end{itemize}
i) Clearly, $\mid\Var(\varphi)\mid < \infty$, hence there exist formulas $\gamma_{1},\dots,\gamma_{n}\in \Gamma$ such that $\Var(\varphi)\subseteq\Var(\gamma_{1})\cup\dots\cup\Var(\gamma_{n})$. Since $\vdash$ is finitary, then there exists a finite set $\Gamma'\subseteq\Gamma$ such that $\Gamma'\vdash\varphi$. If $\Var(\varphi)\subseteq\Var(\Gamma')$, then also $\Gamma'\vdash^{r}\varphi$, i.e. $\vdash^{r}$ is finitary. So, suppose that $\Var(\varphi)\not\subseteq\Var(\Gamma')$. Consider $\Delta=\Gamma'\cup\{\gamma_{1},\dots,\gamma_{n}\}$. Obviously, $\Delta$ is finite, $\Var(\varphi)\subseteq\Var(\Delta)$ and $\Delta\subseteq\Gamma$. By monotonicity of $\vdash$, $\Delta\vdash\varphi$, hence $\Delta\vdash^{r}\varphi$, i.e. $\vdash^{r}$ is finitary. \\
\noindent
ii) Since $\Sigma(x)$ is an antitheorem for $\vdash$, then $\Sigma(x)\vdash\varphi$. Hence \\
\noindent
$\Sigma(x)\vdash_{r}\varphi$ and $\Sigma(x)$ is finite (as $\vdash$ is finite). 
\end{proof}

Since the algebra $\WK$ is a P\l onka sum (of Boolean algebras), it makes sense to ask whether the matrix $\pair{\WK,\{1\}}$ can be constructed as P\l onka sum of (two) matrices. To the best of our knowledge, the construction of P\l onka sums between matrices has been developed exclusively in \cite{BonzioMorascoPrabaldi}. However, it is not difficult to check that the mentioned construction, when applied to the matrices $\pair{\two, \{1\}}$ and $\pair{\mathbf{\ant}, \emptyset}$ (where $\two$ and $\mathbf{\ant}$ stand for the two-element Boolean algebra and the trivial algebra, respectively), does not result in  $\pair{\WK,\{1\}}$. This suggests that a different notion of direct system of logical matrices shall be introduced.

\begin{definition}\label{Def:Directed-System-Matrices}
An \textit{r-direct system} of matrices consists of: 
\benroman
\item A semilattice $I = \langle I, \lor\rangle$.
\item A family of matrices $\{ \langle\A_{i},F_{i}\rangle : i \in I \}$ such that \\ $I^{+}\coloneqq\{i\in I: F_{i}\neq\emptyset\} $ is a  sub-semilattice of $I$.
\item a homomorphism $f_{ij} \colon \A_{i} \to \A_{j}$, for every $i, j \in I$ such that $i \leq j$, satisfying also that: 
\begin{itemize}
\item $f_{ii}$ is the identity map, for every $i \in I$;
\item if $i \leq j \leq k$, then $f_{ik} = f_{jk} \circ f_{ij}$;
\item if $F_{j}\neq\emptyset$ then $ f_{ij}^{-1}[F_{j}]= F_{i}$, for any $i\leq j$.
\end{itemize}
\eroman
\end{definition}

As the nomenclature highlights, the above introduced notion of direct system of matrices is essentially different from the one in \cite{BonzioMorascoPrabaldi}. The main difference concerns the interplay between homomorphisms of the system and matrices' filters.

Given a r-direct system of matrices $X$, we define a new matrix as 
\[
\PL(X) \coloneqq \langle \PL(\A_{i})_{i \in I}, \bigcup_{i \in I}F_{i}\rangle.
\]
\vspace{5pt}

\noindent
We will refer to the matrix $\PL(X)$ as the \textit{P\l onka sum} over the r-direct system of matrices $X$. Given a class $\mathsf{M}$ of matrices, $\PL(\mathsf{M})$ will denote the class of all P\l onka sums of r-direct systems of matrices in $\mathsf{M}$. 

Let $h\colon\Fm\to\PL(\A_i)$ be a homomorphism from the formula algebra into a generic P\l onka sum of algebras. Then, for any formula $\varphi\in Fm$, we set 

\[
 i_{h}(\varphi)\coloneqq\bigvee\{i\in I:h(x)\in A_{i}, x\in\Var(\varphi)\}.
 \]
 
 \vspace{2pt}

In words, $i_{h}(\varphi)$ indicates the index where the formula $\varphi$ is interpreted by the homomorphism $h$, into a P\l onka sum. Moreover, for any $\Gamma\subseteq Fm$, 
we set $i_{h}(\Gamma)\coloneqq\bigvee \{i_{h}(x)\colon x\in\Var(\Gamma)\}$. 

\begin{remark}\label{rem: Gamma finito}
Notice that the index $i_{h}(\Gamma)$ is defined provided that the set $\Var(\Gamma)$ is finite. 
For several results, we will assume that the logic $\vdash$ is finite. Hence, by Lemma \ref{lemma: finitarieta di r}, also $\vdash^{r}$ is finitary, and this allows us to consider finite sets $\Gamma\subseteq Fm$, for which the existence of $i_{h}(\Gamma)$ is assured.
Moreover, observe that, for every homomorphism $h\colon\Fm\to\PL(X)$ from the formula algebra into a generic P\l onka sum over an r-direct system of matrices $X$, and every $\Gamma\cup\{\varphi\}\subseteq Fm$, it is immediate to check that $\Var(\varphi)\subseteq\Var(\Gamma)$ implies $i_{h}(\varphi)\leq i_{h}(\Gamma)$. 
\end{remark}

\begin{lemma}\label{lemma: soundness}
Let $X$ be an r-direct system of non trivial models of a finitary logic $\vdash$. Then $\PL(X)=\langle\PL(\A_i),\bigcup_{i\in I}F_i\rangle$ is a model of $\vdash^{r}$.
\end{lemma}

\begin{proof}

Let $X$ be an r-direct system of non trivial models of $\vdash$. Assume $\Gamma\vdash^{r}\varphi$. Since $\vdash $ is finitary, so it is also $\vdash^{r}$ (by Lemma \ref{lemma: finitarieta di r}), there exists a finite subset $\Delta\subseteq\Gamma $, such that $\Delta\vdash^{r}\varphi$.
We distinguish the following cases:

\begin{itemize}
\item[(a)] $\Sigma(x)\subseteq\Delta$, where $\Sigma(x)$ is an antitheorem of $\vdash$;
\vspace{3pt}
\item[(b)] $\Delta\vdash\varphi$ with $\Var(\varphi)\subseteq\Var(\Delta)$. 
\end{itemize}

\noindent
Since $X$ contains non-trivial models only, the case (a) easily follows by noticing that, for any homomorphism $h\colon \Fm\to\PL(\A_i)$, $h[\Sigma(x)]\not\subseteq F=\bigcup_{i\in I}F_i$.
Therefore $\Sigma(x)\vdash_{\PL(X)}\varphi$, hence also $\Delta\vdash_{\PL(X)}\varphi$. 
\vspace{5pt}

\noindent
Suppose (b) is the case, i.e. $\Delta\vdash\varphi$ with $\Var(\varphi)\subseteq\Var(\Delta)$. Let $h\colon\Fm\to\PL(\A_{i})$ be a homomorphism such that $h[\Delta]\subseteq F$. 
Since $\Delta$ is a finite set, then we can fix $j\coloneqq i_{h}(\Delta)$ and, for any formula $\delta\in \Delta$, we have $h(\delta)\in F_{i_{h}(\delta)}$. This implies that each $i_{h}(\delta)\in I^{+}$ and, as $I^{+}$ forms a  sub-semilattice of $I$, we have that $j\in I^{+}$. 

Now, define $g\colon\Fm\to\A_{j}$ as 
\[
g(x)\coloneqq f_{i_{h}(x)j}\circ h(x),
\]
for every $x\in\Var(\Delta)$. For any $\delta\in \Delta$, we have $g(\delta)=f_{i_{h}(\delta) j}\circ h(\delta)$, hence $g[\Delta]\subseteq F_{j}$. From the fact that $\Delta\vdash\varphi$ and $\langle\A_{j},F_{j}\rangle\in\Mod(\vdash)$, it follows that $g(\varphi)\in F_{j}$. Setting $k\coloneqq  i_{h}(\varphi)$, by Remark \ref{rem: Gamma finito}, we have $k\leq j$ and this, together with the observation that $F_{j}\neq\emptyset$, implies $f_{kj}^{-1}[F_{j}]=F_{k}$. Moreover, we claim that $F_{k}\neq\emptyset$. Suppose, by contradiction, that $F_{k}=\emptyset$. Then, by definition of r-direct system of matrices, we have that $f_{kj}^{-1}[F_{j}]=\emptyset$, that is: there exists no $a\in A_{k}$ such that $f_{kj}(a)\in F_{j}$. On the other hand, since $\Var(\varphi)\subseteq \Var(\Delta)$, then $g(\varphi)=f_{kj}\circ h(\varphi)\in F_{j}$, a contradiction.  

From the fact that $g(\varphi)\in F_{j}$ together with $f_{kj}^{-1}[F_{j}]=F_{k}$, we conclude $h(\varphi)\in F_{k}$. This proves that $h(\varphi)\in F_{k}\subseteq\bigcup_{i\in I}F_{i}$. 
\end{proof}

\begin{remark}\label{rem su inconsistency}

Observe that the assumption on the non-triviality of models of the logic $\vdash$ is crucial in Lemma \ref{lemma: soundness}, as witnessed by the following example. Let $\vdash$ be a theoremless logic possessing an anti-theorem $\Sigma(x)$ (an example is the almost inconsistent logic). 
 Set $X=\pair{\A\oplus \mathbf{1}, A}$ to be the r-direct system of models of $\vdash$, consisting of the two algebras $\mathbf{A}$ and $\mathbf{1}$ with the unique homomorphism $f\colon\A\to\mathbf{1}$ (plus the identity homomorphisms).
Then $\Sigma(x)\vdash y$, for an arbitrary variable $y$, and therefore $\Sigma(x)\vdash^{r} y$. However, $\PL(X)$ is not a model of the latter inference (consider, for instance, an evaluation $v\colon\Fm\to \PL(\A\oplus \mathbf{1})$ such that $v(x)=a\in A$ and $v(y)=1$). 
\end{remark}

Observe that, if the logic $\vdash$ does not possess an antitheorem, then the following holds:
\begin{corollary}
Let $X$ be an r-direct system of models of a finitary logic  $\vdash$ possessing no antitheorems. Then $\PL(X)$ is a model of $\vdash^{r}$.
\end{corollary}

Given a logic $\vdash$ which is complete with respect to a class $\mathsf{M}$ of matrices, we set  $\mathsf{M^{\emptyset}}\coloneqq\mathsf{M}\cup\langle\A,\emptyset\rangle$, for any arbitrary $\A\in\Alg(\vdash)$.


\begin{theorem}\label{completeness}
Let $\vdash$ be a finitary logic which is complete with respect to a class of non trivial matrices $\mathsf{M}$. Then $\vdash^{r}$ is complete with respect to $\PL(\mathsf{M^{\emptyset}})$. 
\end{theorem}

\begin{proof}

We aim at showing that $\vdash^{r}=\;\vdash_{\PL(\mathsf{M^{\emptyset}})}$. \\
\noindent
($\vdash^{r}\;\subseteq \;\vdash_{\PL(\mathsf{M^{\emptyset}})}$). Consider $\Gamma\vdash^r\varphi$ and a P\l onka sum  $\langle\PL (\A_i),\bigcup_{i\in I}F_{i}\rangle$ of matrices in $\mathsf{M^{\emptyset}}$. Set $\A=\PL(\A_i)$ The cases in which $\Gamma$ is an antitheorem of $\vdash$ or $\langle\A,\emptyset\rangle$ is a model of $\vdash$ follow by Lemma \ref{lemma: soundness}. 

So, assume $\langle\A,\emptyset\rangle$ is not a model of $\vdash$ and that $\Gamma$ is not an antitheorem of $\vdash$. 
Let $h\colon Fm\to \A$ be a homomorphism such that $h[\Gamma]\subseteq \bigcup_{i \in I}F_{i}$. Suppose, in view of a contradiction, that $h(\varphi)\not\in \bigcup_{i \in I}F_{i}$. 
 Set $i_{h}(\varphi) = j$ and $ i_{h}(\Gamma) = k$; since $\Var(\varphi)\subseteq\Var(\Gamma)$ then $j\leq k$, by Remark \ref{rem: Gamma finito}. We define a homomorphism $v\colon \Fm\to \A_k$, as follows
\[
v(x)\coloneqq f_{lk}\circ h (x),
\]
where $l=i_{h}(x)$. Clearly, $v[\Gamma]=f_{kk}\circ h[\Gamma] = h[\Gamma]\subseteq F_{k}$ and $v(\varphi)=f_{jk}\circ h (\varphi)\in A_{k}\smallsetminus F_k$, since $h(\varphi)\in A_{j}\smallsetminus F_{j}$ and $F_{j}=f^{-1}_{jk}[F_k]$ (as we know that $F_k\neq\emptyset$). Therefore, we have $\Gamma\not\vdash \varphi$, which is a contradiction. \\
\noindent
($\vdash_{\PL(\mathsf{M^{\emptyset}})}\;\subseteq\;\vdash^{r}$). By contraposition, we prove that $\Gamma\nvdash^{r}\varphi $ implies $\Gamma\nvdash_{\PL(\mathsf{M^{\emptyset}})}\varphi$. To this end, assume $\Gamma\nvdash^{r}\varphi$. If $\Var(\varphi)\subseteq\Var(\Gamma)$, clearly $\Gamma\nvdash\varphi$. Therefore there exists a matrix $\langle\A_{i},F_{i}\rangle\in\mathsf{M}$ and a homomorphism $h\colon\Fm\to\A_{i}$ such that $h[\Gamma]\subseteq F_i$ and $h(\varphi)\notin F_i$. Upon considering the r-direct system $X=\pair{\langle\A_{i},F_{i}\rangle, \{i\}, id}$ and the homomorphism $h$, we immediately obtain $\Gamma\nvdash_{\PL(\mathsf{M^{\emptyset}})}\varphi$.

The only other case to consider is $\Var(\varphi)\nsubseteq\Var(\Gamma)$. Preliminarily, observe that the assumption $\Gamma\nvdash^{r}\varphi$ implies that $\Gamma$ contains no anti-theorem $\Sigma(x)$ for $\vdash$. Therefore, since  $\mathsf{M}$ is a class of models complete with respect to $\vdash$, there exists a matrix $\langle\B,G\rangle\in\mathsf{M}$ and a homomorphism $v\colon\Fm\to\B$ such that $v[\Gamma]\subseteq G$ and $v(\varphi)\notin G$. Consider any r-direct system formed by the matrices $\langle\B,G\rangle$ and  $\langle \A,\emptyset \rangle$ for an appropriate $\A\in\Alg(\vdash)$  (the choice $\A=\boldsymbol{1}$ is always appropriate), with $\langle\A,\emptyset\rangle$  indexed as top element. Denote by $\B \oplus \A^{\emptyset}$ a P\l onka sum over the r-direct system just described. 

The homomorphism $g\colon\Fm\to\B \oplus \A^{\emptyset}$ defined as
\[
g(x)\coloneqq \left\{ \begin{array}{ll}
v(x) \ \text{if} \ \ x\in\Var(\Gamma),& \\
a \ \text{otherwise.} & \\
  \end{array} \right.  
\]
for arbitrary $a\in A$
\noindent
easily witnesses $\Gamma\nvdash_{\PL(\mathsf{M^{\emptyset}})}\varphi$, as desired.
%
%
\end{proof}

\begin{example}\label{ex: containment di PWK}
Let $\mathbf{K_{4}}=\pair{\{0,1,\ant, n \}, \neg,\land,\vee}$ be the algebra given by the following tables 
\vspace{5pt}
\begin{center}\renewcommand{\arraystretch}{1.25}
\begin{tabular}{>{$}c<{$}|>{$}c<{$}}
  \lnot &  \\[.2ex]
\hline
  1 & 0 \\
  \ant & \ant \\
  n & n \\
  0 & 1 \\
\end{tabular}
\qquad
\begin{tabular}{>{$}c<{$}|>{$}c<{$}>{$}c<{$}>{$}c<{$}>{$}c<{$}}
   \land & 0 & \ant & n & 1 \\[.2ex]
 \hline
       0 & 0 & \ant & n &  0 \\
       \ant & \ant & \ant & n & \ant \\
       n & n & n & n & n \\          
       1 & 0 & \ant & n & 1 
\end{tabular}
\qquad
\begin{tabular}{>{$}c<{$}|>{$}c<{$}>{$}c<{$}>{$}c<{$}>{$}c<{$}}
   \lor & 0 & \ant & n & 1 \\[.2ex]
 \hline
       0 & 0 & \ant & n &  1 \\
       \ant & \ant & \ant & n & \ant \\
       n & n & n & n & n \\          
       1 & 1 & \ant & n & 1 
\end{tabular}

\end{center}
\vspace{10pt}

The logic $\mathbf{K^{w}_{4n}}= \pair{\mathbf{K_4},\{1,\ant\}}$ is included among the four-valued regular logics counted by Tomova (see \cite{Tomova,Petrukhin2}). 
Observe that $\pair{\mathbf{K_4},\{1,\ant\}}$ is the P\l onka sum (over the r-direct system) of the matrices $\pair{\mathbf{WK},\{1,\ant\}}$ and $\pair{\mathbf{n},\emptyset}$. Since \PWK is complete with respect to $\pair{\mathbf{WK},\{1,\ant\}}$, then, it follows by Theorem \ref{completeness}, that $\mathbf{K^{w}_{4n}} =\; \vdash^{r}_{\pwk}$, i.e. $\mathbf{K^{w}_{4n}}$ is the containment companion of \PWK. 
\qed
\end{example}

As \PWK is the \emph{left variable inclusion} companion of classical logic (see \cite{BonzioMorascoPrabaldi}), the above example shows that the constructions yielding the (two) companions (left variable inclusion and containment) can actually be iterated, in alternation, starting from an arbitrary logic $\vdash$ (for further details see \cite{prabaldisub}). 

\begin{remark}
Observe that if $\vdash$ is a logic which is complete with respect to a finite set of finite matrices, then so is $\vdash^r$ (by Theorem \ref{completeness}). This means that the containment companion of a logic $\vdash$ preserves ``finite valuedness'', a notion introduced and studied in \cite{CALEIROfinite}. 
\end{remark}

Theorem \ref{completeness} provides a complete class of matrices for an arbitrary logic of right variable inclusion. This class is obtained performing P\l onka sums over r-direct systems of models of $\vdash$ together with the matrices $\langle \A,\emptyset\rangle$ for any $\A\in\Alg(\vdash)$. Obviously, it is not generally the case that the matrix $\langle \A,\emptyset\rangle $ is a model of a logic $\vdash$. 
For this reason, it is not always true that P\l onka sums over an r-direct systems of models of $\vdash$ provide a complete matrix semantics for $\vdash^{r}$. In this sense, the right variable inclusion companion of a logic is a logic of ``P\l onka sums'' (of matrices) in weaker sense compared to the case of the left variable inclusion companion, fully described in \cite{BonzioMorascoPrabaldi}.
Nonetheless, if $\langle \boldsymbol{1},\emptyset\rangle\in\Mod(\vdash)$, the correspondence between $\vdash^{r}$ and P\l onka sums is fully recovered. This is actually the case of every theoremless logic, such as Strong Kleene Logic, $\vdash_{\CL}^{\land,\lor}$. 

\begin{corollary}\label{cor completezza con incons}
A finitary containment logic $\vdash^{r}$ is complete w.r.t. any of the following classes of matrices:
\begin{center}
$\PL(\Mod_{+}(\vdash)\cup\langle \A,\emptyset\rangle)$, $\PL(\Modstar_{+}(\vdash)\cup\langle \A,\emptyset\rangle)$, $\PL(\ModS_{+}(\vdash)\cup\langle \A,\emptyset\rangle),$
\end{center}
for $\A\in\Alg(\vdash)$.
\end{corollary}

\noindent
Moreover, observing that if $\langle \mathbf{1},\emptyset\rangle\in\Mod(\vdash)$ then $\langle \mathbf{1},\emptyset\rangle\in\Modstar(\vdash)$, the following holds

\begin{corollary}
Let $\vdash$ be a finitary logic such that $\langle \mathbf{1},\emptyset\rangle\in\Mod(\vdash)$. Then $\vdash^{r}$ is complete w.r.t. any of the following classes of matrices:
\begin{center}
$\PL(\Mod_{+}(\vdash))$, $\PL(\Modstar_{+}(\vdash))$, $\PL(\ModS_{+}(\vdash)).$
\end{center}
\end{corollary}

In case $\vdash$ does not possess anti-theorems, then the above corollaries can be restated as follows

\begin{corollary}\label{cor da usare per calcolo hilbert no inc}
Let $\vdash$ be a finitary logic without antitheorems. Then $\vdash^{r}$ is complete w.r.t. any of the following classes of matrices:
\begin{center}
$\PL(\Mod(\vdash)\cup\langle\A,\emptyset\rangle)$, $\PL(\Modstar(\vdash)\cup \langle\A,\emptyset\rangle)$, $\PL(\ModS(\vdash)\cup\langle\A,\emptyset\rangle),$
\end{center}
for any $\A\in\Alg(\vdash)$.
\end{corollary}

\begin{corollary}
Let $\vdash$ a finitary logic without antitheorems such that $\langle \boldsymbol{1},\emptyset\rangle\in\Mod(\vdash)$, then $\vdash^{r}$ is complete w.r.t. any of the following classes of matrices:
\begin{center}
$\PL(\Mod(\vdash))$, $\PL(\Modstar(\vdash))$, $\PL(\ModS(\vdash)).$
\end{center}
\end{corollary}



\section{Hilbert style calculi (for logics with r-partition function)}\label{sec:hilbert}

Partitions functions, which have been defined for algebras (see Definition \ref{def: partition function}), can be defined also for logics. 

\begin{definition}\label{def: logic with p-func.}
A logic $\vdash$ has a \emph{r-partition function} if there is a formula $x\ast y$, in which the variables $x$ and $y$ really occur, such that: 
\benroman
\item$x,y\vdash x\ast y$, 
\item$x\ast y \vdash x$,
\item $\varphi(\varepsilon,\vec{z}\?\?)\sineq\varphi(\delta,\vec{z}\?\?)$,
\eroman
for every formula $\varphi(v,\vec{z}\?\?)$ and every identity of the form $\varepsilon \thickapprox \delta$ in Definition \ref{def: partition function}.
\end{definition}

Condition (iii) in the Definition of r-partition function is actually equivalent to say that the term operation $\ast$ is a partition function in every algebra $\A\in\Alg(\vdash)$. This is the consequence of the following (known) fact in abstract algebraic logic. 

\begin{lemma}\label{lemma su equazioni p-function}
Let $\vdash$ be a logic and $\varepsilon,\delta\in Fm $. The following are equivalent:
\benroman
\item $\Alg(\vdash)\vDash\varepsilon\approx\delta$;
\item $\varphi(\varepsilon,\vec{z}\?\?)\sineq\varphi(\delta,\vec{z}\?\?)$, for every formula $\varphi(v,\vec{z}\?\?)$.
\eroman
\end{lemma}
\begin{proof}
See \cite[Lemma 5.74(1)]{Font16} and \cite[Theorem 5.76]{Font16}.
\end{proof}


From now on, we will denote both the formula $x\ast y$ and the term operation $\ast$ as r-partition functions with respect to a logic $\vdash$.

The definition of partition functions for an arbitrary logic is introduced also in \cite[Definition 16]{BonzioMorascoPrabaldi}. It shall be noticed that Definition \ref{def: logic with p-func.} is essentially different (this also motivates the choice of the terminology $r$-partition function). Nevertheless, in most cases (for instance, all substructural logics, classical and modal logics) the very same formula plays both the role of a r-partition function and of a partition function in the sense of \cite[Definition 16]{BonzioMorascoPrabaldi}. 

\begin{example}
Logics with a r-partition function abound in the literature. Indeed, the term $x\ast y\coloneqq x\land(x\lor y)$ is a partition function for every logic $\vdash$ such that $\Alg(\vdash)$ has a lattice reduct. Such examples include all modal and substructural logics \cite{GaJiKoOn07}. On the other hand, the term $x\ast y\coloneqq (y \to y) \to x$ as a r-partition function for all the logics $\vdash$ whose class $\Alg(\vdash)$ possesses a Hilbert algebra (see \cite{Di65}) or a BCK algebra (see \cite{Iseki}) reduct.
\end{example}

\begin{remark}\label{rem: vdash e vdashr hanno le stessa pf}
It is easily checked that a logic $\vdash$ has r-partition function $\ast$ if and only if $\vdash^{r}$ has r-partition function $\ast$.
\end{remark}

In the following, we extend P\l onka representation theorem to r-direct systems of logical matrices. 

\begin{theorem}\label{th: plonka sum of r- matrices}
Let $\vdash$ be a logic with $r$-partition function $\ast$, and $\langle\A,F\rangle$ be a model of $\vdash$ such that $\A\in\Alg(\vdash) $. Then Theorem \ref{th: Teorema di Plonka} holds for $\A$. 
Moreover, by setting $F_{i} \coloneqq F\cap A_{i}$ for every $i \in I$, the triple
\[
X=\langle \langle I,\leq\rangle , \{ \langle \A_{i},F_{i}\rangle\}_{i\in I}, \{ f_{ij} \! : \! i \leq j \}\rangle
\]
is an r-direct system of matrices such that $\PL(X)=\langle\A,F\rangle$.
\end{theorem}

\begin{proof}
Theorem \ref{th: Teorema di Plonka} holds for $\A$, by simply observing that $\ast$ is a partition function for $\A$. 
For the remaining part, it will be enough to show: 
\begin{itemize}
\item[(a)] for every $i,j\in I$ such that $i\leq j$, if $F_{j}\neq\emptyset$ then $f_{ij}^{-1}[F_{j}]=F_{i}$;
\item[(b)]  $I^{+}$ is a sub-semilattice of $I$.
\end{itemize}

In order to prove (a), consider $i,j\in I$ such that $i\leq j$ and let $F_{j}$ be non-empty. Assume, in view of a contradiction, that $f_{ij}^{-1}[F_{j}]\neq F_{i}$. This implies that $F_{i}\nsubseteq f_{ij}^{-1}[F_{j}]$ or that $f_{ij}^{-1}[F_{j}]\nsubseteq F_{i}$. The first case immediately leads to the contradiction that $x,y\nvdash x\ast y$, while the second case contradicts $x\ast y\vdash x$. This proves (a).


In order to prove (b), consider $i,j\in I^{+}$ and let $k= i\lor j$, with $i,j,k\in I$.  As $\ast$ is a $r$-partition function  for $\vdash$, $x,y\vdash x\ast y$. Since $i,j\in I^{+}$, then $F_i$ and $F_j$ are non-empty, therefore there exist two elements $a\in F_{i}$, $b\in F_{j}$. We have $a\ast^{\A}b=f_{ik}(a)\ast^{\A_{k}}f_{jk}(b)\in A_{k}$. This, together with the fact that $\langle\A,F\rangle\in\Mod(\vdash)$ implies $a\ast b\in F_{k}$, i.e. $F_{k}\neq\emptyset$. So $k\in I^{+}$ and this proves (b).
\end{proof}

Given a logic $\vdash$ with a r-partition function $\ast$ and a model $\langle\A,F\rangle$ of $\vdash$ such that $\A\in\Alg(\vdash)$, we call \emph{P\l onka fibers} of $\langle \A, F \rangle$ the matrices $\{ \langle \A_{i},F_{i}\rangle\}_{i\in I}$ given by the decomposition in Theorem \ref{th: plonka sum of r- matrices}. From now on, when considering a model $\pair{\A,F}$ of a logic $\vdash$ with $r$-partition function, we will assume that $\pair{\A, F}= \PL(X)$, for a given direct system $
X=\langle \langle I,\leq\rangle , \{\langle \A_{i},F_{i}\rangle\}_{i\in I}, \{ f_{ij} \! : \! i \leq j \} \rangle
$, without explicitly mentioning the r-direct system $X$.

\begin{lemma}\label{lemma su filtri non vuoti che sono modelli della logica iniziale}
Let $\vdash^{r}$ be a logic with $r$-partition function $\ast$, and $\langle\A,F\rangle\in\Mod(\vdash^{r})$, with $\A\in\Alg(\vdash^{r})$. Then, the P\l onka fibers $ \langle\A_{i},F_{i}\rangle $, such that $i\in I^{+}$, are models of $\vdash$.
\end{lemma}

\begin{proof}

Let $\Gamma\vdash\varphi$ and suppose, by contradiction, that there exists a matrix $\pair{\A_j, F_j}$, with $j\in I^{+}$, and a homomorphism $h\colon\Fm\to\A_{j}$ such that $h[\Gamma]\subseteq F_{j}$ and $h(\varphi)\notin F_{j}$. Preliminarily, observe that $\Var(\varphi)\nsubseteq\Var(\Gamma)$ and, moreover, if $\vdash$ has an antitheorem $\Sigma(x)$, then $\Sigma(x)\nsubseteq\Gamma$, for otherwise $\Gamma\vdash^{r}\varphi$, which is in contradiction with our assumption that  $\langle\A,F\rangle\in\Mod(\vdash^{r})$. Denote by $X$ the (non-empty) set of variables occurring in $\varphi$ but not in $\Gamma$ and, for $\gamma\in \Gamma$, let $X_{\gamma}\coloneqq\{\gamma\ast x: x\in X\}$ and $\Gamma^{-}_{\gamma}\coloneqq \Gamma\smallsetminus\{\gamma\}$. Since $\ast$ is a r-partition function for $\vdash^{r}$, we have $\gamma\ast x\vdash^{r} \gamma$. Therefore $\gamma\ast x\vdash \gamma$ and $X_{\gamma}\vdash \gamma$, which implies $X_{\gamma},\Gamma^{-}_{\gamma}\vdash\varphi$, for any $\gamma\in\Gamma$. Observe that $\Var(\varphi)\subseteq\Var (X_{\gamma})\cup\Var(\Gamma^{-}_{\gamma})$, hence $X_{\gamma},\Gamma^{-}_{\gamma}\vdash^{r}\varphi$.

Since $h(\gamma),h(\varphi)\in A_{j}$ and $x\in\Var(\varphi)$, for every $x\in X$, we have that $h(\gamma\ast x)=h(\gamma)$, whence $h[X_{\gamma}]= h(\gamma)$. Now, for an arbitrary $a\in A$, we define a homomorphism $g\colon\Fm\to\A$, as follows 
\[
g(x)\coloneqq \left\{ \begin{array}{ll}
h(x) \ \text{if} \ \ x\in\Var(\Gamma)\cup\Var(\varphi)& \\
a \ \text{otherwise.} & \\
  \end{array} \right.  
\]

We have $g[X_{\gamma}] = h[X_{\gamma}]= h(\gamma)\in F_j$, $g[\Gamma^{-}_{\gamma}]=h[\Gamma^{-}_{\gamma}]\in F_j$ and $g(\varphi)=h(\varphi)\not\in F_j$. A contradiction.
\end{proof}

The following example provides a simple instance of the above Lemma \ref{lemma su filtri non vuoti che sono modelli della logica iniziale}.

\begin{example}\label{example lemma 24}
Consider the matrix $\pair{\B, F}$ constructed as P\l onka sum of the (non-trivial) Boolean algebras $\A_{i},\A_{j}\A_{k},\A_{s}$ over the r-direct system depicted below (the index set is the four-element Boolean algebra): circles indicate filters, consisting of the top elements $1_i, 1_j$ of $\A_i,\A_j$, respectively. Thus, $F = \{1_{i},1_{j}\}$. Dotted lines represent arbitrary P\l onka homomorphisms.  
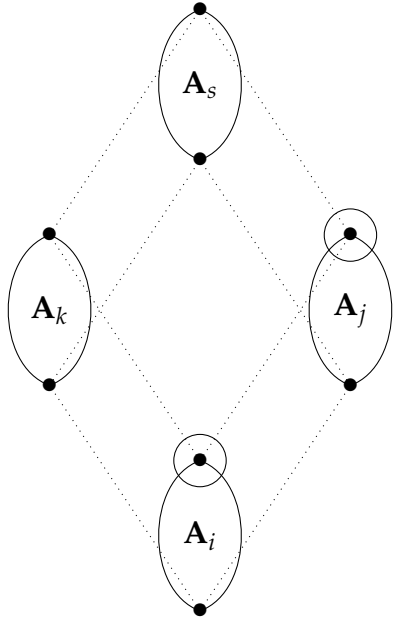
\begin{center}

\begin{tikzpicture}
\draw (-2,6) node {$\bullet$};
\draw (-2,7) node {$ \A_{s}$};
\draw (-2,8) node {$\bullet$};

\draw (-4,5) node {$\bullet$};
\draw (-4,4) node {$ \A_{k}$};
\draw (-4,3) node {$\bullet$};

\draw (0,5) node {\circled{$\bullet$}};
\draw (0,4) node {$ \A_{j}$};
\draw (0,3) node {$\bullet$};

\draw (-2,0) node {$\bullet$};
\draw (-2,1) node {$ \A_{i}$};
\draw (-2,2) node {\circled{$\bullet$}};

\draw (-2,6) edge[bend left=70] (-2,8) ;
\draw (-2,6) edge[bend right=70] (-2,8);
	\draw (-4,5) edge[ dotted] (-2,8) ;
		\draw (-4,3) edge[dotted] (-2,6) ;	
		\draw (0,5) edge[dotted] (-2,8) ;
		\draw (0,3) edge[dotted] (-2,6) ;
		
\draw (-4,5) edge[bend left=70] (-4,3) ;
\draw (-4,5) edge[bend right=70] (-4,3) ;
	\draw (-2,2) edge[dotted] (-4,5) ;
		\draw (-2,0) edge[dotted] (-4,3) ;

\draw (0,5) edge[bend left=70] (0,3) ;
\draw (0,5) edge[bend right=70] (0,3) ;

\draw (-2,0) edge[bend left=70] (-2,2) ;
\draw (-2,0) edge[bend right=70] (-2,2) ;
	\draw (-2,2) edge[dotted] (0,5) ;
		\draw (-2,0) edge[dotted] (0,3) ;

\end{tikzpicture}
\end{center}

By Theorem \ref{completeness}, $\pair{\B, F}$ is a model of Bochvar logic $\mathsf{B_{3}}$. It can be easily checked that $\B\in \Alg(\mathsf{B}_3)$. 
Moreover, $I^+=\{i,j\}$ and clearly $\langle\A_i,1_i\rangle,\langle\A_j,1_j\rangle\in\Mod(\CL)$.
\end{example}

The presence of a r-partition function yields an important syntactic consequence: it allows to adapt a Hilbert style calculus of a logic $\vdash$ into a calculus, for its contaiment companion $\vdash^{r}$. Despite $\vdash^{r}$ is defined via a linguistic restriction constraint (on the inclusion of variables), the axiomatization that we obtain is free of any (linguistic) restriction. 

Throughout the remaining part of this section, we implicitly assume that the logic $\vdash$ possesses an antitheorem. Our analysis can be easily adapted to the case where $\vdash$ does not have antitheorems (see Remark \ref{rem: Hilbert senza inconsistency terms}).

 In what follows, by a \textit{Hilbert-style calculus with finite rules}, we understand a (possibly infinite) set of Hilbert-style rules, each of which has finitely many premises. 

\begin{definition}\label{def: hilbert calc per r}
Let $\mathcal{H}$ be a Hilbert-style calculus with finite rules, which determines a logic $\vdash$ with a $r$-partition function $\ast$ and an antitheorem $\Sigma(x)$.  Let $\mathcal{H}^{r}$ be the Hilbert-style calculus given by the following rules: 

\begin{align}
x\ast \varphi &\rhd \varphi \tag{H0}\label{Eq:Axiom0}\\
x,y &\rhd x\ast y \tag{H1}\label{Eq:Axiom1}\\
x\ast y &\rhd  x\tag{H2}\label{Eq:Axiom2}\\
\{\gamma_{1},\dots,\gamma_{n}\}\smallsetminus\{\gamma_{i}\}, \gamma_{i}\ast\psi &\rhd \psi\tag{H3}\label{Eq:Axiom3} \\
\Sigma(x) &\rhd \alpha \tag{H4}\label{Eq:Axiom4}\\
\chi(\delta, \vec{z}\?\?) \?\?\?\lhd&\rhd \chi(\varepsilon, \vec{z}\?\?)\tag{H5}\label{Eq:Axiom5}
\end{align}

for every
\begin{enumerate}[(i)]
\item $\rhd\varphi$ axiom in $\mathcal{H}$ ;
\item $\gamma_{1},\dots,\gamma_{n}\rhd \psi$ rule in $\mathcal{H}$ (and $\gamma_i$ such that $i\in\{i,\dots,n\}$); 
\item $\delta \thickapprox\varepsilon$ equation in the definition of partition function, and formula $\chi(v, \vec{z}\?\?)$.
\end{enumerate}
\end{definition}

\begin{lemma}\label{lemma: spezza dimostrazione completezza Hilbert}
Let $\vdash$ be a logic with a $r$-partition function $\ast$, an antitheorem $\Sigma(x)$ and let $\langle\A,F\rangle\in\ModS(\vdash_{\mathcal{H}^{r}})$. Then:
\benroman 
\item $\pair{\A,F}=\PL(X)$, where $X=\pair{\langle \langle I,\leq\rangle, \{ \langle \A_{i},F_{i}\rangle\}_{i\in I}, \{ f_{ij} \! : \! i \leq j \}}$ is an r-direct system of matrices;
\item if $X$ contains a trivial matrix then $\A = \mathbf{1}$.
\eroman
\end{lemma}
\begin{proof}
(i) Since $\langle\A,F\rangle\in\ModS(\vdash_{\mathcal{H}^{r}})$, $\A\in\Alg(\vdash_{\mathcal{H}^{r}})$. Moreover, observe that $\ast $ is a $r$-partition function for $\vdash_{\mathcal{H}^{r}}$ (thanks to conditions \eqref{Eq:Axiom1}, \eqref{Eq:Axiom2}, \eqref{Eq:Axiom5}). These facts, together with Theorem \ref{th: plonka sum of r- matrices}, implies that $\langle\A,F\rangle=\PL(X)$, where $X=\langle\{ \langle \A_{i},F_{i}\rangle\}_{i\in I}, \{ f_{ij} \! : \! i \leq j \}, \langle I,\leq\rangle\rangle$ is an r-direct system of matrices.\\
\noindent
(ii) Suppose that, for some $j\in I$, $\langle\A_{j},F_{j}\rangle$ is a trivial fiber of $\langle\A,F\rangle$, i.e. $F_j = A_j$. Since $\Sigma(x)$ is an anti-theorem (for $\vdash$) and \eqref{Eq:Axiom4}  is a rule of $\mathcal{H}^{r}$, then, for every $i\in I$, we have $A_{i}=F_{i}$, i.e. each fiber is trivial. Indeed, if there exists a non trivial fiber $\langle\A_{k},F_{k}\rangle$ and an  element $c\in A_{k}\smallsetminus F_{k}$, then the evaluation $h\colon\Fm\to\A$, defined as $h(x)=a$, $h(y)=c$ (for an arbitrary $a\in\A_{j}$) is such that $h[\Sigma(x)]\subseteq F$ while $h(y)\notin F$, against the fact that $\Sigma(x)\vdash_{\mathcal{H}^{r}}y$. Moreover, the facts that each fiber is trivial and that $\Tarski^{\A}F=id$ immediately implies $\A=\mathbf{1}$.
\end{proof}

 \begin{theorem}\label{th: completezza calcolo Hilbert}
 Let $\vdash$ be a finitary logic with a $r$-partition function $\ast$ and an antitheorem $\Sigma(x)$. Let, moreover, $\mathcal{H}$ be a Hilbert style calculus  with finite rules. If $\mathcal{H}$ is complete for $\vdash$, then $\mathcal{H}^{r}$ is complete for $\vdash^{r}$.
 \end{theorem}
 
 \begin{proof}
 
 Let us denote with $\vdash_{\mathcal{H}^{r}}$ the logic defined by $\mathcal{H}^{r}$. We show that $\vdash_{\mathcal{H}^{r}}\?=\?\vdash^{r}$.
 
 $(\subseteq)$. It is immediate to check that every rule of $\mathcal{H}^{r}$ holds in $\vdash^{r}$.  
 
 $(\supseteq)$. We now show that $\ModS(\vdash_{\mathcal{H}^{r}})\subseteq\Mod(\vdash^{r})$. So let $\langle\A,F\rangle\in\ModS(\vdash_{\mathcal{H}^{r}})$. By Lemma \ref{lemma: spezza dimostrazione completezza Hilbert}-(i), we know that $\langle\A,F\rangle\cong\PL(X)$, where $X=\langle\{ \langle \A_{i},F_{i}\rangle\}_{i\in I}, \{ f_{ij} \! : \! i \leq j \}, \langle I,\leq\rangle\rangle$ is an r-direct system of matrices.
 
The fact that the matrix $\langle\A_{i},F_{i}\rangle\in\Mod(\vdash_{\mathcal{H}})$ for each $i\in I^{+}$ can be proved on the ground of \eqref{Eq:Axiom0} and \eqref{Eq:Axiom3} by adapting the proof of Lemma \ref{lemma su filtri non vuoti che sono modelli della logica iniziale} to the calculus $\mathcal{H}^{r}$. Recalling that $\mathcal{H}$ is complete for $\vdash$ we obtain that $\langle\A_{i},F_{i}\rangle\in\Mod(\vdash)$, for each $i\in I^{+}$. 

%
%
%

\noindent 

Moreover, by Lemma \ref{lemma: spezza dimostrazione completezza Hilbert}-(ii), we know that if $X$ contains a trivial matrix $\pair{\A_j, F_j}$, then $\A = \mathbf{1}$. 

Therefore, two cases may arise: (1) $\A=\mathbf{1}$, (2) $X$ contains no trivial fibers.
If (1) then clearly $\langle\A,F\rangle\in\{\langle\mathbf{1},\emptyset\rangle, \langle\mathbf{1},\{1\}\rangle \}$. As $\vdash^{r}$ is a theoremless logic 
$\{\langle\mathbf{1},\emptyset\rangle, \langle\mathbf{1},1\rangle \}\subseteq \Mod(\vdash^{r})$. If (2), then we can apply Lemma \ref{lemma: soundness}, so $\pair{\A, F}=\PL(X)\in\Mod(\vdash^{r})$. 
 \end{proof}

\begin{remark}\label{rem: Hilbert senza inconsistency terms}
It is easy to check that, if the logic $\vdash$ does not possess antitheorems, then a Hilbert-style calculus for $\vdash^{r}$ can be defined by simply dropping condition (\ref{Eq:Axiom4}) from Definition \ref{def: hilbert calc per r}. The completeness of $\vdash^{r}$ with respect to such calculus can be proven by adapting the strategy in the proof of Theorem \ref{th: completezza calcolo Hilbert}.
\end{remark}

\section{Examples of axiomatizations}\label{sec: examples}

In this last section, we show how to obtain Hilbert-style axiomatizations of some containment logics.  

\subsection{Bochvar logic}

Bochvar logic is the containment companion of classical logic. Consider the following Hilbert-style axiomatization of classical propositional logic:
\benroman
\item[($\mathbf{H}_{1}$)] $\rhd\varphi\to\varphi$
\item [($\mathbf{H}_{2}$)]$\rhd\varphi\to(\psi\to\varphi)$
\item [($\mathbf{H}_{3}$)]$\rhd\varphi\to(\psi\to\chi)\to(\varphi\to\psi)\to(\varphi\to\chi)$
\item [($\mathbf{H}_{4}$)]$\rhd(\neg\varphi\to\neg\psi)\to(\psi\to\varphi)$
\item [($\mathbf{H}_{5}$)]$\varphi,\varphi\to\psi\rhd\varphi$
\eroman

 Theorem \ref{th: completezza calcolo Hilbert} allows to provide the following complete Hilbert style calculus for Bochvar logic $\mathsf{B_{3}}$.

\begin{itemize}
\item[($\mathbf{H}^{r}_{1}$)] $\alpha\ast(\varphi\to\varphi)\rhd\varphi\to\varphi$
\item [($\mathbf{H}^{r}_{2}$)]$\alpha\ast(\varphi\to(\psi\to\varphi)\rhd\varphi\to(\psi\to\varphi)$
\item [($\mathbf{H}^{r}_{3}$)]$\alpha\ast(\varphi\to(\psi\to\chi)\to(\varphi\to\psi)\to(\varphi\to\chi))\rhd\varphi\to(\psi\to\chi)\to(\varphi\to\psi)\to(\varphi\to\chi)$
\item[($\mathbf{H}^{r}_{4}$)] $\alpha\ast((\neg\varphi\to\neg\psi)\to(\psi\to\varphi))\rhd(\neg\varphi\to\neg\psi)\to(\psi\to\varphi)$
\item [($\mathbf{H}^{r}_{5}$)]$\varphi\ast\psi,\varphi\to\psi\rhd\psi$
\item [($\mathbf{H}^{r}_{6}$)] $\alpha,\neg \alpha\rhd\varphi$
\item[($\mathbf{H}^{r}_{7}$)] $\chi(\delta, \vec{z}\?\?) \?\?\?\lhd\rhd\chi(\varepsilon, \vec{z}\?\?)$ for every formula $\chi(x,\vec{z})$ and equation $\delta\thickapprox\varepsilon$ in  Definition \ref{def: partition function},
\end{itemize}

where $\varphi\ast\psi$ is an abbreviation for $\varphi\land(\varphi\lor\psi)$.

\subsection{Belnap-Dunn logic}

Belnap-Dunn is a four-valued logic $\mathsf{B}$, originally introduced as \emph{First Degree Entailment} within the research enterprise on relevance and entailment logic \cite{AndersonBelnap,Belnap1977}.   

Consider the algebraic language of type $1,2,2$, containing $\neg,\vee,\land$. Recall that a \emph{De Morgan lattice} is an algebra $\A=\pair{A,\neg,\vee,\land}$ of type $1,2,2$ such that: 
\benroman
\item $\pair{A,\land,\vee}$ is a distributive lattice; 
\item $\neg $ satisfies the following equations: 
\[
x\approx\neg\neg x, \hspace{1cm} \neg(x\land y)\approx \neg x\lor\neg y, \hspace{1cm} \neg(x\lor y)\approx \neg x\land\neg y.
\]
\eroman

De Morgan lattices, originally introduced by Moisil \cite{Moisil} and, independently, by Kalman \cite{Kalman58} (under the name of \emph{distributive i-lattices}) form a variety, which is generated by the four element algebra $\M=\pair{\{0,b,n,1\},\neg,\vee,\land}$, whose lattice reduct is displayed in Figure \ref{fig: M4} and negation in the following table:
\begin{center}\renewcommand{\arraystretch}{1.25}
\begin{tabular}{>{$}c<{$}|>{$}c<{$}}
  \lnot &  \\[.2ex]
\hline
  1 & 0 \\
  b & b \\
  n & n \\
  0 & 1 \\
\end{tabular}
\end{center}

\begin{center}
\begin{figure}[h]
\begin{tikzcd}[row sep = tiny, arrows = {dash}]
 & 1 & \\
 & & \\
 & & \\
 b \arrow[uuur, dash] & & n \arrow[uuul] \\
 & & \\
 & & \\
 & 0\arrow[uuul]\arrow[uuur] &
  \end{tikzcd}\caption{Hasse
  diagram of the De Morgan lattice $\M$.}
 \end{figure}
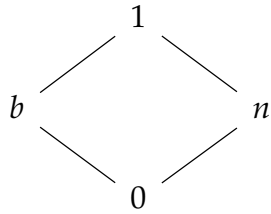\label{fig: M4}
 
  \end{center}

$\mathsf{B}$ is the logic induced by the matrix $\pair{\M,\{1,b\}}$ (or, equivalently, by $\pair{\M,\{1,n\}}$, see \cite[Proposition 2.3]{font1997belnap}). $\mathsf{B}$ is finitary and theoremless (purely inferential). Moreover, the class $\Alg(\mathsf{B})$ coincides with the variety of De Morgan lattices \cite[Theorem 4.1]{font1997belnap}. Observe that the set $\{\varphi,\neg\varphi\}$ is not an antitheorem of $\mathsf{B}$ (indeed $\varphi,\neg\varphi\not\vdash_{\mathsf{B}}\psi$). It is not difficult to check that $\mathsf{B}$ does not possess antitheorems.  

Recall that a \emph{lattice filter} of a De Morgan lattice $\A$ is a subset $F\subseteq A$ such that $x\land y\in F$ if and only if $x\in F$ and $y\in F$. The class of matrices 
\[
\mathsf{M}=\{\pair{\A, F} : \; \A \text{ De Morgan algebra, } F\subseteq A \text{ lattice filter}\}
\]

is complete for $\mathsf{B}$ \cite[Corollary 2.6]{font1997belnap}. Observe that the matrices $\pair{\A,\emptyset}\in\mathsf{M}$, for any De Morgan lattice $\A$. It follows from our analysis (see Theorem \ref{completeness}) that the containment companion $\vdash^{r}_{\mathsf{B}}$ of $\vdash_{\mathsf{B}}$ is complete with respect to $\PL{(\mathsf{M})}$, i.e. the class of all P\l onka sums over r-direct systems of matrices in $\mathsf{M}$.

We present the Hilbert-style axiomatization for $\mathsf{B}$ which is introduced in \cite{font1997belnap} (and, independently in \cite{Pynko95}). Since $\mathsf{B}$ is theoremless, the calculus has no axioms and the following rules: 

\benroman
\item[($\mathsf{B}_1$)] $\varphi \land \psi \rhd \varphi $;
\item [($\mathsf{B}_{2}$)] $\varphi \land \psi \rhd \psi$;
\item [($\mathsf{B}_{3}$)]$\varphi, \psi\rhd \varphi  \land \psi $;
\item [($\mathsf{B}_{4}$)]$\varphi \rhd \varphi\lor \psi$;
\item [($\mathsf{B}_{5}$)]$\varphi\lor \psi\rhd \psi\lor \varphi$;
\item [($\mathsf{B}_{6}$)]$\varphi\lor \varphi \rhd \varphi$;
\item [($\mathsf{B}_{7}$)]$\varphi\lor (\psi\lor\chi) \rhd (\varphi\lor\psi)\lor\chi$;
\item [($\mathsf{B}_{8}$)]$\varphi\lor (\psi\land\chi) \rhd (\varphi\lor\psi)\land(\varphi\lor\chi)$;
\item [($\mathsf{B}_{9}$)]$(\varphi\lor\psi)\land(\varphi\lor\chi) \rhd  \varphi\lor (\psi\land\chi)$;
\item [($\mathsf{B}_{10}$)]$\varphi\lor\psi  \rhd \neg\neg\varphi\lor\psi$;
\item [($\mathsf{B}_{11}$)]$\neg\neg\varphi\lor\psi  \rhd \varphi\lor\psi$
\item [($\mathsf{B}_{12}$)]$\neg(\varphi\lor\psi)\lor\chi  \rhd (\neg\varphi\land\neg \psi)\lor\chi$;
\item [($\mathsf{B}_{13}$)]$ (\neg\varphi\land\neg \psi)\lor\chi  \rhd \neg(\varphi\lor\psi)\lor\chi$;
\item [($\mathsf{B}_{14}$)]$\neg(\varphi\land\psi)\lor\chi  \rhd (\neg\varphi\lor\neg \psi)\lor\chi$;
\item [($\mathsf{B}_{15}$)]$ (\neg\varphi\lor\neg \psi)\lor\chi  \rhd \neg(\varphi\land\psi)\lor\chi$;
\eroman

A Hilbert-style axiomatization of $\vdash^{r}_{\mathsf{B}}$, (see Definition \ref{def: hilbert calc per r} Theorem \ref{th: completezza calcolo Hilbert}) is given by the following ($\varphi\ast\psi$ is an abbreviation for $ \varphi\land(\varphi\lor\psi)$): 

\benroman
\item[($\mathsf{B}^{r}_{1}$)] $\varphi,\psi \rhd \varphi \ast \psi $;
\item[($\mathsf{B}^{r}_{2}$)] $\varphi\ast \psi \rhd \psi $;
\item[($\mathsf{B}^{r}_{3}$)] $(\varphi \land \psi)\ast\varphi \rhd \varphi $;
\item [($\mathsf{B}^{r}_{4}$)] $(\varphi \land \psi)\ast\psi \rhd \psi$;
\item [($\mathsf{B}^{r}_{5}$)]$\varphi\ast (\varphi  \land \psi) , \psi\rhd \varphi  \land \psi $;
\item [($\mathsf{B}^{r}_{6}$)]$\varphi, \psi\ast(\varphi  \land \psi)\rhd \varphi  \land \psi $;
\item [($\mathsf{B}^{r}_{7}$)]$\varphi\ast(\varphi\lor \psi) \rhd \varphi\lor \psi$;
\item [($\mathsf{B}^{r}_{8}$)]$(\varphi\lor \psi)\ast (\psi\lor \varphi)\rhd \psi\lor \varphi$;
\item [($\mathsf{B}^{r}_{9}$)]$(\varphi\lor \varphi)\ast\varphi \rhd \varphi$;

\item [($\mathsf{B}^{r}_{10}$)]$(\varphi\lor (\psi\lor\chi))\ast ((\varphi\lor\psi)\lor\chi) \rhd (\varphi\lor\psi)\lor\chi$;

\item [($\mathsf{B}^{r}_{11}$)]$\varphi\lor (\psi\land\chi)\ast((\varphi\lor\psi)\land(\varphi\lor\chi)) \rhd (\varphi\lor\psi)\land(\varphi\lor\chi)$;

\item [($\mathsf{B}^{r}_{12}$)]$((\varphi\lor\psi)\land(\varphi\lor\chi))\ast(\varphi\lor (\psi\land\chi)) \rhd  \varphi\lor (\psi\land\chi)$;

\item [($\mathsf{B}^{r}_{13}$)]$(\varphi\lor\psi)\ast (\neg\neg\varphi\lor\psi)  \rhd \neg\neg\varphi\lor\psi$;

\item [($\mathsf{B}^{r}_{14}$)]$(\neg\neg\varphi\lor\psi)\ast (\varphi\lor\psi)\rhd \varphi\lor\psi$;

\item [($\mathsf{B}^{r}_{15}$)]$(\neg(\varphi\lor\psi)\lor\chi)\ast ((\neg\varphi\land\neg \psi)\lor\chi)  \rhd (\neg\varphi\land\neg \psi)\lor\chi$;

\item [($\mathsf{B}^{r}_{16}$)]$ ((\neg\varphi\land\neg \psi)\lor\chi) \ast (\neg(\varphi\lor\psi)\lor\chi)  \rhd \neg(\varphi\lor\psi)\lor\chi$;

\item [($\mathsf{B}^{r}_{17}$)]$(\neg(\varphi\land\psi)\lor\chi)\ast ((\neg\varphi\lor\neg \psi)\lor\chi)  \rhd (\neg\varphi\lor\neg \psi)\lor\chi$;

\item [($\mathsf{B}^{r}_{18}$)]$ ((\neg\varphi\lor\neg \psi)\lor\chi)\ast (\neg(\varphi\land\psi)\lor\chi)  \rhd \neg(\varphi\land\psi)\lor\chi$;

\item[($\mathsf{B}^{r}_{19}$)] $\chi(\delta, \vec{z}\?\?) \?\?\?\lhd\rhd\chi(\varepsilon, \vec{z}\?\?)$ for every formula $\chi(x,\vec{z})$ and equation $\delta\thickapprox\varepsilon$ in  Definition \ref{def: partition function}.
\eroman

\subsection{The Relevance logic $\mathrm{\mathbf{S}_{fde}}$}\label{Subsec: relevance}

The logic $\mathrm{\mathbf{S}_{fde}}$ has been introduced by Deutsch \cite{Deutsch}: it is induced the matrix $\pair{\mathbf{S}_{4}, \{1,\ant\}}$, whose algebraic reduct $\mathbf{S}_{4}=\pair{\{0,\ant,m,1\}, \neg, \land, \vee}$ is given in the following tables.
\vspace{5pt}
\begin{center}\renewcommand{\arraystretch}{1.25}
\begin{tabular}{>{$}c<{$}|>{$}c<{$}}
  \lnot &  \\[.2ex]
\hline
  1 & 0 \\
  \ant & \ant \\
  m & m \\
  0 & 1 \\
\end{tabular}
\qquad
\begin{tabular}{>{$}c<{$}|>{$}c<{$}>{$}c<{$}>{$}c<{$}>{$}c<{$}}
   \land & 0 & \ant & m & 1 \\[.2ex]
 \hline
       0 & 0 & 0 & m &  0 \\
       \ant & 0 & \ant & m & \ant \\
       m & m & m & m & m \\          
       1 & 0 & \ant & m & 1 
\end{tabular}
\qquad
\begin{tabular}{>{$}c<{$}|>{$}c<{$}>{$}c<{$}>{$}c<{$}>{$}c<{$}}
   \lor & 0 & \ant & m & 1 \\[.2ex]
 \hline
       0 & 0 & \ant & m &  1 \\
       \ant & \ant & \ant & m & 1 \\
       m & m & m & m & m \\          
       1 & 1 & 1 & m & 1 
\end{tabular}

\end{center}
\vspace{10pt}
\noindent
Recall that a \emph{Kleene lattice} is a De Morgan lattice satisfying $x\land\neg x\leq y\lor\neg y$. Kleene lattices form a variety ($\mathsf{KL}$), generated by the 3-element algebra $\mathbf{SK}=\pair{\{0,1,\ant\}, \neg, \vee, \land}$, which is a subalgebra of $\mathbf{S}_{4}$ (and also isomorphic to the two three-element subalgebras of $\M$).

The logic of Paradox $\mathsf{LP}$ (see \cite{Priestfirst,BlRiVe01,Pynko}) is defined by the matrix $\langle\mathbf{SK},\{1,\ant\}\rangle$. The algebraic counterpart of $\mathsf{LP}$ is exactly the variety of Kleene lattice, i.e. $\mathsf{KL}=\Alg(\mathsf{LP})$. 
It is immediate to check that the matrix $\pair{\mathbf{S}_{4}, \{1,\ant\}}$ is the P\l onka sum over the r-direct system of the two matrices $\pair{\mathbf{SK},\{1,\ant\}}$ and $\pair{\mathbf{m},\emptyset}$. Thus, it follows from Theorem \ref{completeness} that $\mathrm{\mathbf{S}_{fde}}$ is the containment companion of the logic of Paradox. 

A finite Hilbert style calculus for $\mathsf{LP}$ (see for instance \cite{Superbelnap}) can be obtained by adding the axiom

\benroman
\item[$(\mathsf{LP}_{1})$] $\rhd \varphi\lor\neg\varphi$
\eroman
to the calculus for the logic $\mathsf{B}$ described above. Therefore, by Theorem \ref{th: completezza calcolo Hilbert}, the calculus consisting of $(\mathsf{B}_1^r)-(\mathsf{B}_{19}^r)$ and
\benroman
\item [$(\mathsf{LP}^r_{1})$] $\alpha\ast \varphi\lor\neg\varphi\rhd \varphi\lor\neg\varphi $
 
\eroman

is complete for $\mathrm{\mathbf{S}_{fde}}$.

%
%

\end{document}